\documentclass[11pt,reqno,a4paper]{amsart}
\usepackage{mathtools}
\usepackage{etoolbox}
\usepackage{amsmath}
\usepackage{enumerate}
\usepackage{amssymb}
\usepackage{amscd} 
\usepackage{amsthm}
\usepackage{amsfonts}
\usepackage{cancel}
\usepackage[usenames,dvipsnames,svgnames]{xcolor}
\usepackage{tikz}
\usepackage{graphicx}
\usepackage[matrix, arrow, curve]{xy}
\usepackage{comment}
\usepackage[scaled]{helvet} % ss                                                     
\usepackage{courier} % tt        
\usepackage[T1]{fontenc}                                                    
\usepackage{hyperref}
\usepackage{imakeidx}
\usepackage{extarrows}
\usepackage[utf8]{inputenc}

\numberwithin{equation}{section}

\newcommand{\cU}{\mathcal{U}}

\newcommand{\N}{\mathbb{N}}

\newcommand{\R}{\mathbb{R}}
\newcommand{\Z}{\mathbb{Z}}

\newcommand{\asdim}{\operatorname{asdim}}

\newcommand{\diam}{\operatorname{diam}}

\newcommand{\qand}{\quad \textrm{and} \quad}

\DeclareMathOperator{\ind}{ind}

\theoremstyle{thm}
\newtheorem{theorem}{Theorem}[section]
\newtheorem{corollary}[theorem]{Corollary}
\newtheorem{proposition}[theorem]{Proposition}
\newtheorem{lemma}[theorem]{Lemma}

\newenvironment{usethmcounterof}[1]{%
  \renewcommand\thetheorem{\ref{#1}}\theorem}{\endtheorem\addtocounter{theorem}{-1}}

\theoremstyle{definition}
\newtheorem{definition}[theorem]{Definition}

\newtheorem{remark}[theorem]{Remark}
\newtheorem{example}[theorem]{Example}

\patchcmd{\section}{-.5em}{.5em}{}{}
\patchcmd{\subsubsection}{-.5em}{.5em}{}{}

\makeatletter
\@namedef{subjclassname@2020}{%
  \textup{2020} Mathematics Subject Classification}
\makeatother

\makeatletter
\let\@wraptoccontribs\wraptoccontribs
\makeatother

\begin{document}

%date
%\noindent 4 November 2025

\title[Hurewicz-type theorem for quasimorphisms of countable approximate groups]{A Hurewicz-type theorem for quasimorphisms of countable approximate groups}

%    Author information
\author{Vera Toni\'c}
\address{Faculty of Mathematics, University of Rijeka, Croatia}
\curraddr{}
\email{vera.tonic@math.uniri.hr}
\thanks{...}



\date{\today}

\begin{abstract} 
 In their theorem from 2006, A.~Dranishnikov and J.~Smith (\cite{DranSmith}) prove that if $f:G\to H$ is a group homomorphism, then the following formula for asymptotic dimension is true: $\asdim G \leq \asdim H + \asdim\ \! (\ker f)$. This result is known as the Hurewicz-type formula, after a 1927 theorem from classical dimension theory by W.~Hurewicz, which inspired it.

In this paper we establish a similar formula to the one by Dranishnikov and Smith, for the following setup: 
 whenever $(\Xi, \Xi^\infty)$ and $(\Lambda,\Lambda^\infty)$ are countable approximate groups and  $f:(\Xi, \Xi^\infty) \to (\Lambda,\Lambda^\infty)$ is a (general) quasimorphism, i.e., a quasimorphism which need not be symmetric nor unital, then the following formula is true:
$$\asdim \Xi \leq \asdim \Lambda + \asdim\ \! \left(f^{-1}\left( f(e_\Xi)D(f)^{-1}D(f)\right)\right), $$
where $D(f)$ is the defect set of $f$. 
It follows as a corollary that if $f:G\to H$ is a quasimorphism of countable groups, then $\asdim G\leq \asdim H +  \asdim\ \! \left(f^{-1}\left( f(e_\Xi)D(f)^{-1}D(f)\right)\right)$.
\end{abstract}

\maketitle

\section{Introduction}
The result in this paper is motivated by a wish to generalize a well-known
Hurewicz-type formula for asymptotic dimension and homomorphisms of groups, due to A.~Dranishnikov and J.~Smith
 (\cite[Theorem 2.3]{DranSmith}), which states:
\begin{theorem}[Dranishnikov and Smith]
\label{Thm: Hurewicz DS}
Let $f:G\to H$ be a  group homomorphism. Then
$\asdim G \leq \asdim H + \asdim\ \! (\ker f)$.
\end{theorem}
Their theorem is one in the sequence of results that began with W.~Hurewicz's theorem from 1927, known as \emph{dimension-lowering mapping theorem}, stating that for a continuous map $f:X\to Y$ between compact metric spaces it is true that $\dim X\leq \dim Y +\dim f$, where $\dim$ is referring to the (Lebesgue) covering dimension\footnote{Hurewicz's dimension theorems were originally proven for the small inductive dimension $\ind$.}, and $\dim f=\sup \{\dim (f^{-1}(y)) \mid y\in Y\}$. According to \cite{Engelking}, this was followed by
the proof of the same formula for separable metric spaces and closed maps (by W.~Hurewicz and H.~Wallman, 1941), and then for metrizable spaces and closed maps (K.~Morita, 1956, and separately K.~Nagami, 1957, \cite[Theorem 4.3.4]{Engelking}):
\begin{theorem}[Morita, Nagami] \label{MoritaNagami}
Let $f: X \to Y$ be a closed map of metrizable spaces. Then
$\dim X \leq \dim Y + \dim f$, where $\dim f \coloneqq  \sup\ \!\{ \dim(f^{-1}(y))\ | \  y\in Y\}$.
\end{theorem}

In the 2000's, in addition to Theorem \ref{Thm: Hurewicz DS} stated above, statements similar to Theorem \ref{MoritaNagami} were proven for asymptotic dimension, first by G.~Bell and A.~Dranishnikov (\cite[Theorem 1]{BellDran-Hurewicz} or \cite[Theorem 29]{BellDran1}), and then generalized by N.~Brodskiy, J.~Dydak, M.~Levin and A.~Mitra (\cite[Theorem 1.2]{BDLM}):

\begin{theorem}[Bell and Dranishnikov]
\label{AbstractHurewicz0} Let $f: X \to Y$ be a  Lipschitz map from a geodesic metric space $X$ to a metric space $Y$.  If for every $r>0$ the collection
$\mathbb X_r \coloneqq  \{f^{-1}(B(y, r))\}_{y \in Y}$ satisfies 
$\asdim\ \!(\mathbb X_r)  \leq n$ uniformly,
then $\asdim X \leq \asdim Y + n.$
%\qed
\end{theorem}

\begin{theorem}[Brodskiy, Dydak, Levin and Mitra]\label{AbstractHurewicz} 
Let $f: X \to Y$ be a coarsely Lipschitz map between metric spaces.  Then $\asdim X \leq \asdim Y + \asdim f,$
where $\asdim f\coloneqq \sup \{ \asdim A \mid A \subseteq X \text{ and } \asdim f(A) = 0\}.$
\end{theorem}
 
We will define all notions mentioned in Theorem \ref{AbstractHurewicz0} and Theorem \ref{AbstractHurewicz} in Section 2. Note that in the original
statement of Theorem \ref{AbstractHurewicz} in \cite{BDLM}, the authors were referring to \emph{large scale uniform maps}, but these maps are called \emph{coarsely Lipschitz} in this paper, following the convention of \cite{CHT}.

By T.~Hartnick and the author of this paper, Theorem \ref{AbstractHurewicz}  was used to prove that an analogous formula to the one in Theorem \ref{Thm: Hurewicz DS} works for asymptotic dimension and global morphisms of countable approximate groups (\cite[Theorem 1.4]{Hartnick-Tonic}): 
\begin{theorem}
\label{Thm: H-T} Let $(\Xi, \Xi^\infty)$ and $(\Lambda, \Lambda^\infty)$ be countable approximate groups and let $f: (\Xi, \Xi^\infty) \to (\Lambda, \Lambda^\infty)$ be a global morphism. Then
 $$
 \asdim\Xi \leq \asdim\Lambda + \asdim\ \!([\![\ker(f)]\!]_c),
 $$
where $[\![\ker(f)]\!]_c$ is the \emph{coarse kernel} of $f$, defined by 
$
[\![\ker (f)]\!]_c\coloneq [\Xi^2\cap\ker f]_c 
$, and $[\cdot]_c$ stands for the coarse class of the (sub)space within brackets.
\end{theorem}

Theorem \ref{AbstractHurewicz}   was also used to prove the similar formula for symmetric unital quasimorphisms of countable approximate groups (\cite[Theorem 4.6]{Tonic}):
\begin{theorem}\label{Thm-symm-quasimorphism}
Let $(\Xi, \Xi^\infty)$ and $(\Lambda,\Lambda^\infty)$ be countable approximate groups, and let $f:(\Xi, \Xi^\infty) \to (\Lambda,\Lambda^\infty)$ be a symmetric unital quasimorphism. Then 
\[ \asdim \Xi \leq \asdim \Lambda + \asdim\ \! (f^{-1}(D(f))),\]
where $D(f)$ is the defect set of $f$.
\end{theorem}

The goal of this paper is to extend Theorem \ref{Thm-symm-quasimorphism} by proving the following

\begin{theorem}\label{main}
Let $(\Xi, \Xi^\infty)$ and $(\Lambda,\Lambda^\infty)$ be countable approximate groups, and let $f:(\Xi, \Xi^\infty) \to (\Lambda,\Lambda^\infty)$ be a (general) quasimorphism. Then
$$\asdim \Xi \leq \asdim \Lambda + \asdim \left(f^{-1}( f(e_\Xi)D(f)^{-1}D(f))\right), $$
where $D(f)$ is the defect set of $f$. 
\end{theorem}
 
Note that $D(f)^{-1}$ is the set of inverses of all elements od $D(f)$, and that a general quasimorphism mentioned in Theorem \ref{main} need not be symmetric nor unital.
We prove this theorem in the last section, Section $5$, in which we also phrase a corollary of Theorem \ref{main} (Corollary \ref{cor main}), stating that whenever $f:G\to H$ is a quasimorphism of countable groups, then $\asdim G\leq \asdim H +  \asdim\ \! \left(f^{-1}\left( f(e_\Xi)D(f)^{-1}D(f)\right)\right)$. 
In order to get to the proof of Theorem \ref{main}, we first introduce asymptotic dimension, coarse equivalences and coarsely Lipschitz maps in Section $2$. Section $3$ contains the basic facts on approximate groups and on quasimorphisms between groups, and between approximate groups. Section $4$ is devoted to a lemma connecting general quasimorphisms of countable approximate groups with coarsely Lipschitz maps, allowing us to use Theorem \ref{AbstractHurewicz} to finish the proof of Theorem \ref{main} in Section~5.

\section{Coarsely Lipschitz maps and asymptotic dimension of metric spaces}
We use the words \emph{map} and \emph{function} interchangeably, though perhaps the word ``function'' is used more often when its domain and codomain are subsets of $\R$ or $\R\times \R$. Also, for $R>0$, we use the common notation $N_R(A)$ for an open $R$-neighborhood of a set $A$ in a metric space $(X,d)$, as well as $B(x, R)$ for an open ball and $\overline B(x,R)$ for a closed ball centered at a point $x$ and with radius $R$. If we need to emphasize the metric, we write $B_d(x, R)$ and $\overline B_d(x,R)$.

\begin{definition}\label{def: coarsely - maps} 
Let $(X, d_X)$, $(Y, d_Y)$ be metric spaces and let $f: X \to Y$ be a map. 
\begin{enumerate}[(i)]
\item If there exists a non-decreasing function $\Phi_+: [0, \infty) \to [0, \infty)$ such that $\lim_{t \to \infty} \Phi_+(t) = \infty$, and such that
$d_Y(f(x), f(x')) \leq \Phi_+(d_X(x,x'))$ for all $x,x'\in X$, then we say that $f$ is \emph{coarsely Lipschitz}.

\item If there exist non-decreasing functions $\Phi_-, \ \Phi_+: [0, \infty) \to [0, \infty)$ such that $\lim_{t \to \infty} \Phi_i(t) = \infty$, for $i=-,+$, and such that
$\Phi_-(d_X(x,x')) \leq d_Y(f(x), f(x')) \leq \Phi_+(d_X(x,x'))$ for all $x,x'\in X$, then we say that $f$ is a \emph{coarse embedding}.

\item If $f$ is a coarse embedding which is also \emph{coarsely surjective}, i.e., if there exists a $C\geq 0$ such that $N_C(f(X))=Y$, then $f$ is called a \emph{coarse equivalence}.
\end{enumerate}
\end{definition}

Two metric spaces $X$ and $Y$ are said to be \emph{coarsely equivalent} if there is a coarse equivalence between them, in which case we write $X \overset{CE}{\approx} Y$, and we also say that they belong to the same \emph{coarse class} (of spaces). Properties of metric spaces that are preserved by coarse equivalences are referred to as \emph{coarse invariants}.

Coarsely Lipschitz maps can be characterized without mentioning a function $\Phi_+$ (see \cite[Prop.\ 3.A.5]{CdlH}):
\begin{lemma} \label{characterization-CL}
A map $f:X\to Y$ between metric spaces is coarsely Lipschitz if and only if for every $t> 0$ there exists an $s> 0$ such that whenever $x, x' \in X$ satisfy $d_X(x, x')\leq t$, then $d_Y (f(x),f(x'))\leq s$. 
\end{lemma}

\medskip

Let us now state the definition for asymptotic dimension of a metric space, known as the \emph{coloring definition} of $\asdim$ (for other equivalent definitions of $\asdim$, see, for example \cite{BellDran1}).

\begin{definition}\label{def: asdim}
Let $n\in\N_0$. We say that a metric space \emph{$(X, d)$ has asymptotic dimension at most $n$}, and we write $\asdim X\leq n$, if for every $r>0$ there is a cover $\cU$ of $X$ such that:
\begin{enumerate}
\item $\cU$ can be written as $\cU=\bigcup_{i=1}^{n+1} \cU_i$, which is a union of $n+1$ subfamilies where each subfamily $\cU_i$ is $r$-disjoint, that is, if $U$ and $U'$ are two different elements of $\cU_i$, then $d(U,U')\geq r$, and
\item $\cU$ is uniformly bounded, i.e., there exists a $D>0$ such that $\diam U\leq D$, for all $U\in \cU$.
\end{enumerate}
\end{definition}

Since this $D$  in Definition \ref{def: asdim} depends on the choice of $r$, we can introduce a function $D_X:(0, \infty) \to (0, \infty)$ by $D_X(r):=D$, and rephrase Definition \ref{def: asdim} as follows:
for $n\in \N_0$, we say that $\asdim X\leq n$ if there is a function $D_X:(0, \infty) \to (0, \infty)$ such that for every $r>0$ there is a cover 
 $\cU=\bigcup_{i=1}^{n+1} \cU_i$ of $X$ such that each subfamily $\cU_i$ is $r$-disjoint and the diameter of every element of $\cU$ is bounded above by $D_X(r)$.

We refer to this function $D_X$ as an \emph{$n$-dimensional control function of $X$}.

We say that $\asdim X=n$ if this $n$ is the smallest number satisfying Definition \ref{def: asdim}, and if there is no such $n\in \N_0$, we say that $\asdim X=\infty$.

For detailed properties of asymptotic dimension, see, for example, \cite{BellDran1}. We will mention some that are relevant in this paper.

\begin{lemma}\label{lem: asdim properties}
\begin{enumerate}
\item  If $(X,d_X)$ is a metric space and $A\subseteq X$, then $\asdim A \leq \asdim X$.
\item If $(X,d_X)$ and $(Y, d_Y)$ are coarsely equivalent metric spaces, then $\asdim X =\asdim Y$, i.e., $\asdim$ is a coarse invariant.
\end{enumerate}
\end{lemma}
\begin{corollary}\label{cor: asdim prop}
If $(X,d_X)$ is a metric space and $R>0$, then for any $A\subseteq X$, $\asdim N_R(A)=\asdim A$.
\end{corollary}
\begin{proof}
The inclusion $A\hookrightarrow N_R(A)$ is a coarse equivalence.
\end{proof}
\medskip

Next we introduce the notion of a family of subsets of a metric space satisfying the same upper bound for their $\asdim$ \emph{uniformly} (\cite[Section 2]{BellDranOnAsdimOfGroups}):
\begin{definition}\label{def: uniform asdim}
Let $n\in\N_0$, let $J$ be some indexing set and let $\mathbb{X}=\{X_\alpha \ | \ \alpha\in J\}$ be a family of subsets of a metric space $(X,d)$. We say that \emph{the asymptotic dimension of the family $\mathbb{X}$ is uniformly bounded by $n$} (or that $\mathbb{X}$ satisfies $\asdim \mathbb{X}\leq n$ \emph{uniformly}), and write $\asdim \mathbb{X}\overset{u}{\leq} n$, if for any $r>0$ there exists $D>0$ such that for every $\alpha \in J$ there is a cover $\cU_\alpha$ of $X_\alpha$ which satisfies the two properties of Definition \ref{def: asdim} with respect to these $r$ and $D$. That is, there exists a function $D_\mathbb{X}:(0,\infty)\to (0,\infty)$ serving as an $n$-dimensional control function simultaneously for all $X_\alpha$, $\alpha\in J$.
\end{definition}

\medskip

Next we introduce the notion of an \emph{$n$-dimensional control function for a map} between metric spaces (\cite{BDLM} Definition 4.4).

\begin{definition}\label{def: control of function}
Let $f:(X,d_X)\to (Y,d_Y)$ be a map between metric spaces and let $n\in\N_0$. An \emph{$n$-dimensional control function of $f$} is a function $D_f:(0, \infty)\times(0, \infty) \to (0,\infty)$ such that for all $r_X>0$ and $R_Y>0$ we have the following: for any subset $A\subseteq X$ with $\diam f(A)\leq R_Y$, this $A$ can be expressed as the union of $n+1$ sets with $r_X$-components of these sets being $D_f(r_X,R_Y)$-bounded.

The $r_X$-components of a set $B\subseteq X$ are the maximal $r_X$-connected subsets of $B$, where a set $C$ is $r_X$-connected if any two elements $x,x'\in C$ can be connected in $C$ by an $r_X$-chain, that is, there exists a sequence of points $x=x_1, x_2, \ldots ,x_{k-1},x_k=x'$ in $C$ such that $d_X(x_i,x_{i+1})\leq r_X$, for all $i<k$.
\end{definition}

Note that $\diam f(A)\leq R_Y$ means that $f(A)\subseteq \overline{B}_{d_Y}(f(a), R_Y)$, for any $a\in A\subseteq X$. 

\medskip

In the next lemma and its proof we will use shorter notation for open balls $B(x, R):=B_{d_X}(x, R)$  in $(X,d_X)$ and $B'(y, R):=B_{d_Y}(y,R)$ in $(Y,d_Y)$, for $x\in X$, $y\in Y$, $R>0$.

\begin{lemma}\label{lemma: control}
Let $n\in \N_0$ and let a map $f:(X,d_X)\to (Y,d_Y)$ have the property that, for each $R>0$, 
$$\asdim \{f^{-1}(B'(f(x), R)) \ | \ x\in X\} \overset{u}{\leq} n.$$
Then $f$ has an $n$-dimensional control function. 
\end{lemma}

Before proceeding, note that in the proof of Corollary 4.12 of \cite{BDLM} it is mentioned that having,  for each $R>0$, 
$\asdim \{f^{-1}(B'(y, R)) \ | \ y\in Y\} \overset{u}{\leq} n$
implies that $f$ has an $n$-dimensional control function, but since we are using $f(x)$ instead of $y\in Y$ for centers of our open balls in the formula above, and $f$ need not be surjective, let us prove Lemma \ref{lemma: control}, to be safe.

\begin{proof}[Proof of Lemma \ref{lemma: control}]
According to Definition \ref{def: uniform asdim}, the assumption of this lemma means that, given any $R_Y>0$, there exists a function $D_{R_Y}:(0, \infty)\to(0,\infty)$ such that for each $r_X>0$ and $x\in X$, there is a cover $\cU_x=\bigcup_{i=1}^{n+1} \cU_x^{(i)}$ of $f^{-1}(B'(f(x), R_Y))$ such that each subfamily $\cU_x^{(i)}$ is $r_X$-disjoint and $D_{R_Y}(r_X)$-bounded, i.e., all elements of $\cU_x$ have diameter $\leq D_{R_Y}(r_X)$.

Let us define a function $D_f:(0,\infty)\times(0,\infty)\to (0,\infty)$ by $D_f(r_X, R_Y):=D_{R_Y}(r_X)$.\\
Now we can reformulate the assumption of this lemma as follows: there is a function $D_f:(0,\infty)\times(0,\infty)\to (0,\infty)$ such that for any $x\in X$, for any $r_X>0$ and any $R_Y>0$ we have that $f^{-1}(B'(f(x), R_Y)$ has a cover $\cU_x=\bigcup_{i=1}^{n+1} \cU_x^{(i)}$ such that each subfamily $\cU_x^{(i)}$ is $r_X$-disjoint and $D_f(r_X, R_Y)$-bounded.

We would like to show that for a function $\varphi:(0,\infty)\times(0,\infty)\to (0,\infty)\times(0,\infty)$ given by $\varphi(t,s)=(2t,2s)$ we get that $D_f\circ \varphi$ is an $n$-dimensional control function of our initial map $f:X\to Y$. 

So let us take any $r_X>0$, $R_Y>0$ and let $A\subseteq X$ be any subset with $\diam f(A)\leq R_Y$, i.e., such that for all $a, a'\in A$ we have $d_Y(f(a),f(a'))\leq R_Y$. This implies that $f(A)\subseteq \overline{B'} (f(a), R_Y)\subseteq B'(f(a), 2R_Y)$, for some $a\in A\subseteq X$. Therefore $A\subseteq f^{-1}(f(A))\subseteq f^{-1}(B'(f(a),2R_Y))$.

We can now use our reformulated assumption for $2r_X>0$ and $2R_Y>0$, that is, for the family of subsets 
$\{f^{-1}(B'(f(x), 2R_Y)) \ | \ x\in X\}$: we know that for each $x\in X$ there is a cover $\cU_x=\bigcup_{i=1}^{n+1} \cU_x^{(i)}$ of $f^{-1}(B'(f(x), 2R_Y))$ such that each subfamily $\cU_x^{(i)}$ is $2r_X$-disjoint and $D_f(2r_X, 2R_Y)$-bounded. So in particular, for $A\subseteq f^{-1}(B'(f(a),2R_Y))$ there is a cover  $\cU_a=\bigcup_{i=1}^{n+1} \cU_a^{(i)}$ of $f^{-1}(B'(f(a),2R_Y))$ such that each subfamily 
$\cU_a^{(i)}$ is $2r_X$-disjoint and $D_f(2r_X, 2R_Y)$-bounded. Therefore $A$ can be expressed as the union of $n+1$ sets, $A= \bigcup_{i=1}^{n+1} A\cap(\cup \cU_a^{(i)})$, where for each two different elements $U,U'\in\cU_a^{(i)}$ we have $d_X(U,U')\geq 2r_X$ and also $\diam U\leq D_f(2r_X,2R_Y)$, for all $U\in \cU_a$. This means that each $r_X$-component of every set $A\cap(\cup \cU_a^{(i)})$ is contained within some $A\cap U$, for some $U\in \cU_a^{(i)}$, so each $r_X$-component of every set $A\cap(\cup \cU_a^{(i)})$ is bounded by the upper bound of all $\diam U$, for $U\in \cU_a$, that is, each $r_X$-component of $A\cap(\cup \cU_a^{(i)})$ has diameter $\leq D_f(2r_X,2R_Y)$. Therefore, by Definition \ref{def: control of function}, $D_f\circ \varphi$ is an $n$-dimensional control function of our initial map $f:X\to Y$. 
\end{proof}

Recall that in Theorem \ref{AbstractHurewicz} we mentioned $\asdim f$, which is defined by $\asdim f\coloneqq \sup \{ \asdim A \mid A \subseteq X \text{ and } \asdim f(A) = 0\}.$ Corollary 4.10 from \cite{BDLM} states:

\begin{lemma}\label{le: asdim f vs control}
If $f:X\to Y$ is a coarsely Lipschitz map of metric spaces and $n\in\N_0$, then $\asdim f \leq n$ if and only if $f$ has an $n$-dimensional control function.
\end{lemma}

Therefore we can state the following corollary of Theorem \ref{AbstractHurewicz}, Lemma \ref{lemma: control} and Lemma \ref{le: asdim f vs control}:
\begin{corollary}\label{cor: AbstractHurewicz}
If $n\in \N_0$ and $f:X\to Y$ is a coarsely Lipschitz map between metric spaces with the property that, for each $R>0$, 
$\asdim \{f^{-1}(B'(f(x), R)) \ | \ x\in X\} \overset{u}{\leq} n$, then $\asdim X \leq \asdim Y + n$.
\end{corollary}
\begin{proof}
First, from Lemma  \ref{lemma: control} it follows that $f$ has an $n$-dimensional control function. Since $f$ is also coarsely Lipschitz, having an $n$-dimensional control function is equivalent to $\asdim f\leq n$, by Lemma \ref{le: asdim f vs control}. This and Theorem \ref{AbstractHurewicz} give us the final statement.
\end{proof}

\section{Approximate groups and quasimorphisms}
%%%%%
We start with a reminder about notation: if $A$ and $B$ are subsets of a group $(G,\cdot)$, then $AB=A\cdot B:=\{a b\ | \ a\in A,\ b\in B\}$. In particular,  $A^2=A\cdot A=\{ab\ | \ a,b\in A\}$, and $A^k=A^{k-1}A$, for $k\in\N_{\geq 2}$. Also, we use $A^{-1}:=\{a^{-1}\ | \ a\in A\}$, and if $A=A^{-1}$, we say that $A$ is \emph{symmetric}.
The identity element of the group $G$ is written as $e$ or $e_G$, and if $e\in A$, we say that $A$ is \emph{unital}. For $g\in G$, we use 
$gA=g\cdot A:=\{g a\ | \ a\in A\}$. We will sometimes make the operation sign visible to make our formulas easier to understand.

Also, in sections that follow we encounter functions between groups and the subsets of groups, so if $(G,\cdot_G)$ and $(H, \cdot_H)$ are groups, $A$ is a symmetric subset of $G$ and $f:A\to H$ is a set-theoretic function, we say that $f$ is \emph{symmetric} if $f(a^{-1})=f(a)^{-1}$, for all $a\in A$. If $A$ is unital, we say that $f:A\to H$ is \emph{unital} if $f(e_G)=e_H$.

In addition, we need to be cautious about notation on functions: if $(G,\cdot_G)$ and $(H, \cdot_H)$ are groups, $A\subseteq G$, $D\subseteq C\subseteq H$ are some of their subsets and $f:A\to C$ is a function, we need to carefully distinguish between the way we write preimages and inverses. For an element $h\in C\subseteq H$, $f^{-1}(h)$ is the $f$-preimage of the element $h$, so $f^{-1}(h)\subseteq A\subseteq G$,
while for an element $g\in A$, $f(g)^{-1}$ is  the inverse of the element $f(g)\in C\subseteq H$, so $f(g)^{-1}\in H$. 
 More generally, note that $f^{-1}(D)=\{g\in A\ | \ f(g)\in D\}\subseteq A$, while $f(A)^{-1}=\{f(g)^{-1}\ | \ g\in A\}\subseteq H$.
 
 \bigskip

Let us now define approximate subgroups and approximate groups. The following definition for an approximate subgroup of a group is 
 due to T.~Tao (\cite{Tao}):

\begin{definition}\label{DefTao} Let $(G,\cdot)$ be a group and let $k \in \N$. A subset $\Lambda \subseteq G$ is called 
a \emph{$k$-approximate subgroup} of $G$ if
\begin{enumerate}
\item $\Lambda = \Lambda^{-1}$ and $e \in \Lambda$, that is, $\Lambda$ is symmetric and unital, and
\item there exists a finite subset $F \subseteq  G$ such that $\Lambda^2\subseteq \Lambda  F$ and $|F| = k$.
\end{enumerate}
We say that $\Lambda$ is an \emph{approximate subgroup} of $G$ if it is a $k$-approximate subgroup of $G$, for some $k \in \N$. 
\end{definition}

Let us compare this with subgroups: if $H$ is a subgroup of a group $G$, then $H$ is symmetric, unital, and closed under the group operation so $H^2=H$, while for an approximate subgroup $\Lambda$ of $G$, this ``$\Lambda^2=\Lambda$'' is spoiled, but not by much, since $\Lambda^2$ is only \emph{finite set away} from $\Lambda$.
We will focus on countable groups and their infinite approximate subgroups, so while we need the set $F$ from Definition \ref{DefTao} to be finite, the number $k$ of elements of $F$ will not be relevant. Note that for an approximate subgroup $\Lambda$ there is the smallest subgroup $\Lambda^\infty := \bigcup_{k \in \N} \Lambda^k$ of $G$ which contains $\Lambda$, which plays an important role:

\begin{definition}\label{DefApGr}
Let $\Lambda$ be an approximate subgroup of a group $G$. Then  $\Lambda^\infty = \bigcup_{k \in \N} \Lambda^k$ is called the \emph{enveloping group} of $\Lambda$, and the pair $(\Lambda, \Lambda^\infty)$ is called an \emph{approximate group}.
\end{definition}
We say that the approximate group $(\Lambda, \Lambda^\infty)$ is \emph{countable} if $\Lambda$ is countable, which also implies that $\Lambda^\infty$ is countable. We say that $(\Lambda, \Lambda^\infty)$ is \emph{finite} if $\Lambda$ is finite.
The notation  $(\Lambda, \Lambda^\infty)$ is used because we will need a metric on $\Lambda$, which we will introduce by taking a ``nice enough'' metric on $\Lambda^\infty$. However, the asymptotic dimension of $(\Lambda, \Lambda^\infty)$ will be defined as $\asdim$ of $\Lambda$. But before discussing our choice of metrics and introducing $\asdim$ of approximate groups, let us mention some examples (from \cite{CHT} or \cite{Hartnick-Tonic}):

\begin{example} Let $G$ be a group. Then obviously every subgroup $H$ of $G$ is an approximate subgroup of $G$, and $(H,H)$ is an approximate group. Also, any finite symmetric unital subset $F$ of $G$ is an approximate subgroup of $G$, so $(F, F^\infty)$ is an approximate group. Moreover,
if $\Lambda$ is an approximate subgroup of $G$, then $\Lambda^k$ is also an approximate subgroup of $G$, so $(\Lambda^k,\Lambda^\infty)$ is an approximate group, for all $k\in\N$.
\end{example}

\begin{example}
 For the  Baumslag-Solitar group ${\rm BS}(1,2) =  \langle a, b \mid bab^{-1} = a^2 \rangle$, the set
$\Lambda :=  \langle a \rangle \cup \{b,b^{-1}\}$ is symmetric, unital, and $ \Lambda^2 \subseteq \Lambda\cdot \{e, b, b^{-1}, b^{-1}a\}$, so $\Lambda$ is an approximate subgroup of ${\rm BS}(1,2)$, hence $(\Lambda, \Lambda^\infty)=(\Lambda, {\rm BS}(1,2))$ is an approximate group.
\end{example}

\begin{example} 
Let $G$ and $H$ be locally compact groups, 
and let $\Gamma$ be a subgroup of $G \times H$ such that the restriction $\pi_G|_\Gamma$ is injective, where $\pi_G:G\times H\to G$ is the canonical projection. 
If $W$ is any relatively compact symmetric neighborhood of identity in $H$, take all elements of $\Gamma$ within the ``strip'' $G \times W$ in $G \times H$, and project these elements to $G$. Then $\Lambda(\Gamma, W) := \pi_G(\Gamma \cap (G \times W))$ is an approximate subgroup of $G$, so $(\Lambda(\Gamma, W), (\Lambda(\Gamma, W))^\infty)$ is an approximate group.
This construction of $\Lambda(\Gamma, W)$ is referred to as \emph{cut-and-project} construction.
\end{example}

\medskip

Let us proceed by choosing a ``nice enough'' metric on a countable group $\Lambda^\infty$. In \cite[Section 1]{DranSmith} it is explained how we can always choose a left-invariant proper metric on a countable group (which need not be finitely generated), 
so that this metric agrees with discrete topology on the group.
The properness of the metric means that the closed balls are compact, and since we are in a discrete group, all (open or closed) balls with a bounded radius are finite sets.
Moreover, \cite[Prop. 1.1]{DranSmith} gives us:
\begin{proposition} If $d$ and $d'$ are two left-invariant proper metrics on the same countable group $G$, then the identity map from
$(G, d)$ to $(G, d')$ is a coarse equivalence. %\qed
\end{proposition}

Therefore, we have:
\begin{corollary}\label{ExternalQIType} 
If $d$ and $d'$ are two left-invariant proper metrics on the same countable group $G$, and if $A \subseteq G$ is any subset, then the identity map from $(A, d|_{A \times A})$ to $(A, d'|_{A \times A})$ is a coarse equivalence. %\qed
\end{corollary}

In consequence, we define:
\begin{definition}\label{def: canonical coarse class}
For a countable group $G$,
its \emph{canonical coarse class $[G]_c$} is the coarse class
of the metric space $(G, d)$, where $d$ is some (hence any) left-invariant proper metric on $G$, i.e., 
$[G]_c:=[(G,d)]_c=\{(X,d')\ |\ (X,d') \text{ is a metric space such that } (X,d')\overset{CE}{\approx} (G, d)\}.$ We say that any left-invariant proper metric $d$ on $G$ is a  \emph{canonical metric} on $G$.

\medskip

For a countable approximate group $(\Lambda, \Lambda^\infty)$ and for any subset $A \subseteq \Lambda^\infty$, the \emph{canonical coarse class of $A$} is defined by
$$
[A]_c \coloneqq  [(A, d|_{A \times A})]_c = \{(X,d')\ |\ (X,d') \text{ is a metric space such that } (X,d')\overset{CE}{\approx} (A, d|_{A \times A})\},
$$
where $d$ is some (hence any) left-invariant proper metric on the countable group $\Lambda^\infty$.
In particular, this defines the canonical coarse class of $\Lambda$,
$[\Lambda]_c \coloneqq  [(\Lambda, d|_{\Lambda \times \Lambda})]_c$.
For any left-invariant proper metric $d$ on $\Lambda^\infty$, we refer to the metric $d|_{\Lambda\times\Lambda}$ as a \emph{canonical metric} on $\Lambda$.
\end{definition}

\medskip

We now define the asymptotic dimension of a countable group and any of its subsets, as well as $\asdim$ of a countable approximate group:
\begin{definition}\label{def: asdim - group}
For a countable group $G$, we define its asymptotic dimension by
\[\asdim G := \asdim \ \! (G, d),\] 
where $d$ is any left-invariant proper metric on $G$. We also define $\asdim \ \! [G]_c:= \asdim \ \! (G, d)$.

For a countable approximate group $(\Lambda, \Lambda^\infty)$, we define its asymptotic dimension by
\[\asdim \Lambda := \asdim \ \! (\Lambda, d|_{\Lambda \times \Lambda}),
\]
where $d$ is any left-invariant proper metric on $\Lambda^\infty$. We also define $\asdim \ \! [\Lambda]_c := \asdim \ \! (\Lambda, d|_{\Lambda\times\Lambda})$.
More generally, if $A$ is any subset of $\Lambda^\infty$,  we define the asymptotic dimension of $A$ by $\asdim A := \asdim \ \! (A, d|_{A \times A})=:\asdim\ \! [A]_c$.
\end{definition}

\bigskip
Our next goal is to define quasimorphisms of groups and approximate groups. Behind this definition is an idea of S.~Ulam (\cite{Ulam}), where a set-theoretic function $f:G\to H$ between groups need not satisfy $f(xy)=f(x)f(y)$ for all $x,y\in G$, but $f$ should ``approximately'' satisfy it by being a ``finite number of elements away'' from it, which is illustrated by the following definition (\cite[Definition 2.38]{CHT}).

\begin{definition}\label{DefQM} A function $f: G \to H$ between groups $G$ and $H$ is called a \emph{quasimorphism} if its \emph{defect set}
\begin{equation}\label{LDefectSet}
D(f) \coloneqq  \{f(y)^{-1}f(x)^{-1}f(xy) \mid x, y \in G\}
\end{equation}
 is finite.
\end{definition}

If $f:G\to H$ was a group homomorphism, that is, satisfying $f(xy)=f(x)f(y)$ for all $x,y\in G$, then we would have $f(y)^{-1}f(x)^{-1}f(xy)=e_H$, for all $x,y\in G$, but Definition \ref{DefQM} allows the set $D(f)=\{f(y)^{-1}f(x)^{-1}f(xy) \mid x, y \in G\}$ to contain more than one element, as long as this set is finite. Note that, in general, $D(f)$ is not symmetric nor unital.

\begin{remark} The function with property from Definition \ref{DefQM} could be called a \emph{left-quasimor\-phism}, while $D(f)$ could be called a \emph{left-defect set}, and we could define a \emph{right-quasimorphism} by asking that the \emph{right-defect set}
\begin{equation}\label{RDefectSet}
D^*(f) \coloneqq  \{f(x)f(y)f(xy)^{-1} \mid x, y \in G\}
\end{equation}
should be finite. However, N.~Heuer has shown in \cite[Prop.\ 2.3]{Heuer1} that the two notions coincide.
\end{remark}

\begin{remark}\label{Drho} Using
$f(xy) = f(x)f(y)f(y)^{-1}f(x)^{-1}f(xy)$, we get that the defect set $D(f)$ has the following properties:
\begin{equation}\label{Drho1}
f(xy) \in f(x)f(y)D(f) \qand f(x)f(y) \in f(xy)D(f)^{-1}, \quad \text{ for all }x,y \in G.
\end{equation}
Also, from the second part of \eqref{Drho1}, we have that
\begin{equation}\label{Drho-1}
 f(y)^{-1}f(x)^{-1}\in D(f)f(xy)^{-1}, \quad \text{ for all }x,y \in G.
\end{equation}
\end{remark}

\medskip
If $H$ is a countable group with a left-invariant proper metric $d$, then a subset of $(H,d)$ is finite if and only if it is contained in a ball of bounded radius centered at $e_H$. Therefore we have
 (\cite[Proposition 2.47]{CHT}):
\begin{proposition}\label{prop: quasi alter}
If $G$ is a group and $H$ is a countable group with a left-invariant proper metric $d$, then a function $f:G\to H$ is a quasimorphism if and only if there exists a constant $C\geq 0$ such that 
$$
d(f(xy),f(x)f(y))=d( f(y)^{-1}f(x)^{-1}f(xy),e_H)\leq C, \ \ \text{ for all } x,y \in G.
$$
\end{proposition}

\begin{remark} There is some variety of terminology regarding the words ``quasimorphism'' and ``quasihomomorphism''. Throughout this paper we are using terminology of \cite{CHT}, i.e., this is where
the name \emph{quasimorphism} in Definition \ref{DefQM} is coming from. Other sources, like \cite{FujiKap} or \cite{Heuer1}, use the name \emph{quasihomomorphism} for the kind of function from Definition \ref{DefQM}.  In fact, in \cite{Heuer1}, \cite{FujiKap}, or \cite{Rolli}, for example, the name \emph{quasimorphism} is used for a function $f:G\to \R$ from a group $G$ to $\R$ for which there is a $C\geq 0$ such that $|f(xy)-f(x)-f(y)|\leq C$, for all $x,y\in G$. In \cite{Fujiwara1998} or \cite{BestvinaFujiwara} such functions $f:G\to \R$ are called \emph{quasi-homomorphisms}, while in \cite{CHT} such functions are referred to as \emph{real-valued topological quasimorphisms}. If we replace the codomain $\R$ of $f$ with $\Z$, this definition becomes equivalent to Definition \ref{DefQM}, by Proposition \ref{prop: quasi alter}. More about real-valued topological quasimorphisms can be found in \cite[Appendix B.5]{CHT}.
\end{remark}

Let us now look at some examples of quasimorphisms.

\begin{example}\label{ex: QM} Clearly, all group homomorphisms are also quasimorphisms. All functions between groups which have finite image are quasimorphisms. If $G$ is a group, $f:G\to \Z$ is a homomorphism and $b:G\to \Z$ is a bounded function, then $f+b:G\to \Z$ is a quasimorphism.
 In addition (see, for example, \cite{Heuer1}), new quasimorphisms can be constructed by taking any existing quasimorphism $\phi: G\to \Z$ from a group $G$ to $\Z$, taking a group $H$ which contains an infinite cyclic subgroup $K$ and taking a homomorphism $\tau:\Z \to H$ such that $\tau(\Z)=K$, to produce the composition $\tau\circ\phi:G\to H$, which is then a quasimorphism.
\end{example}

\begin{example}\label{ex: Rolli Brooks}
Let us see two types of quasimorphisms defined on free groups $F_r$ of rank $r\geq 2$, with codomain $\Z$. First, here is an adapted version of
P.~Rolli's quasimorphisms (\cite{Rolli}): let $S=\{s_1, \ldots , s_r\}$ be a free generating set of a free group $F_r$ of rank $r\geq 2$. For each function $\alpha:\Z\to \Z$ which is bounded and odd (i.e., $\alpha(-n)=-\alpha(n)$, for all $n\in \Z$), define a function $f_\alpha:F_r \to \Z$ by $f_\alpha(s_{i_1}^{n_1}\ldots  s_{i_k}^{n_k})\coloneqq \sum_{j=1}^{k} \alpha(n_j)$, using the fact that each non-trivial reduced word  $w$ in $F_r$ has a unique shortest factorization into powers $w=s_{i_1}^{n_1}\ldots  s_{i_k}^{n_k}$, for $s_{i_j}\in S$, $n_j\in\Z$. Then this $f_\alpha$ is a quasimorphism.

Secondly, we define Brooks' \emph{counting quasimorphisms} (\cite{Brooks}): as in the previous example, start with a free generating set $S=\{s_1, \ldots , s_r\}$ for $F_r$ of rank $r\geq 2$, and
choose some reduced word $w\in F_r\setminus\{e\}$. Then, for any reduced word $v\in F_r$, let $c_w(v)$ be the number of appearances of the word $w$ in $v$ and $c_{w^{-1}}(v)$ be the number of appearances of $w^{-1}$ in $v$. Then the function $h_w:F_r\to \Z$ defined by $h_w(v)\coloneqq c_w(v) - c_{w^{-1}}(v)$  is a quasimorphism (see, for example, \cite[Example B.43]{CHT} for more details). 
\end{example}
The following example is a generalization of Brooks' counting quasimorphisms (see \cite{Fujiwara1998} or \cite{BestvinaFujiwara}).
\begin{example}\label{ex: Fujiwara}
Let $X$ be a geodesic metric space, and let $G$ be a group acting on $X$ properly discontinuously by isometries. For a finite path 
$\alpha$ in $X$, let $i(\alpha)$, $t(\alpha)$ mark the initial and terminal points of $\alpha$, and let $\vert \alpha \vert$ be the length of $\alpha$. Use the action of $g\in G$ on $X$ to define a path $g(\alpha)$ with $g(i(\alpha))$, $g(t(\alpha))$ as its initial and terminal points (respectively), and note that $\vert g(\alpha)\vert =\vert\alpha\vert$.
Now let $w$ be a path in $X$ with $\vert w\vert \geq 2$. For any path $\alpha$ in $X$ define $\vert\alpha\vert_w$ to be the maximal number of non-overlapping copies of $w$ in $\alpha$. Also, whenever $x,y$ are points in $X$, denote by $\left[ x,y\right]$ some choice of a geodesic from $x$ to $y$, and define 
$$c_w(\left[ x,y\right])\coloneqq d(x,y) -\inf_\alpha (\vert\alpha\vert -\vert \alpha\vert_w),$$
where $\alpha$ ranges over all paths from $x$ to $y$. Finally, fix a base point $x_0\in X$ and define a function $h_w:G\to \Z$ by
$$h_w(g)\coloneqq c_w(\left[x_0,g(x_0)\right]) - c_{w^{-1}}(\left[x_0,g(x_0)\right]).$$
By \cite[Proposition 3.10]{Fujiwara1998}, when $X$ is Gromov-hyperbolic, this $h_w$ is a quasimorphism.

A special case of this construction can be found in \cite{EpsteinFujiwara}, where $G$ is a hyperbolic group and $X$ is its Cayley graph with respect to some finite generating set, and with $x_0=e$ as the base point.
\end{example}

\medskip

More examples of quasimorphisms can be found in  \cite{FujiKap} and  \cite{BrandVerbitsky}: \cite{FujiKap} deals with quasimorphisms that have discrete groups as codomains, while \cite{BrandVerbitsky} covers a more general theory of so-called \emph{Ulam quasimorphisms}.

\medskip

Recall that we would like to define quasimorphisms between approximate groups. First, let us define \emph{global morphisms}:

\begin{definition}
Let $(\Xi, \Xi^\infty)$ and $(\Lambda, \Lambda^\infty)$ be approximate groups.
A \emph{global morphism} $f: (\Xi, \Xi^\infty) \to (\Lambda, \Lambda^\infty)$ between approximate groups is a group homomorphism $f: \Xi^\infty \to \Lambda^\infty$ which restricts to partial homomorphisms $f_k:=f|_{\Xi^k} :\Xi^{k}\to \Lambda^k$ for each $k\in \N$, that is, for all $\xi_1, \xi_2 \in \Xi^k$ which satisfy $\xi_1\xi_2 \in \Xi^k$ we have $f_k(\xi_1\xi_2) = f_k(\xi_1)f_k(\xi_2)$.
Therefore we have that  $f(\Xi^k)=f(\Xi)^k\subseteq \Lambda^k$, for all $k \in \N$.
\end{definition}

\begin{definition} \label{def: QM}
Let $(\Xi, \Xi^\infty)$ and $(\Lambda, \Lambda^\infty)$ be approximate groups. A \emph{global quasimorphism} (or simply a \emph{quasimorphism}) between approximate groups is
 a function of pairs $f: (\Xi, \Xi^\infty) \to (\Lambda, \Lambda^\infty)$ 
such that $f: \Xi^\infty \to \Lambda^\infty$ is a quasimorphism in the sense of Definition \ref{DefQM}.
\end{definition}
Note that the defect set of a quasimorphism $f: (\Xi, \Xi^\infty) \to (\Lambda, \Lambda^\infty)$ is $D(f)$ of $f:\Xi^\infty\to \Lambda^\infty$.
Also, a  quasimorphism $f: (\Xi, \Xi^\infty) \to (\Lambda, \Lambda^\infty)$  of approximate groups satisfies $f(\Xi)\subseteq\Lambda$, and we use notation $f_1:=f|_\Xi:\Xi \to \Lambda$.
However, $f(\Xi^2)$ need not be contained in $\Lambda^2$, but since $D(f)$ is finite and therefore contained in $\Lambda^M$ for some $M \in \N$, from \eqref{Drho1} we get that $f(\Xi^2)\subseteq f(\Xi)^2D(f)\subseteq \Lambda^{M+2}$.

Every global morphism of approximate groups is an example of a symmetric and unital quasimorphism between them. 
Another way to produce examples of quasimorphisms of approximate groups is to take quasimorphisms between groups and restrict to approximate subgroups.

%%%%%%%%
\section{Quasimorphisms and being coarsely Lipschitz}
What follows is a version of Lemma 3.8 from \cite{CHT}, stated for (general) quasimorphisms between countable approximate groups.

\begin{lemma}\label{coarsely-Lipschitz}
Let $(\Xi, \Xi^\infty)$ and $(\Lambda,\Lambda^\infty)$ be countable approximate groups and let $f:(\Xi, \Xi^\infty) \to (\Lambda,\Lambda^\infty)$ be a quasimorphism. Then the restriction $f_1=f|_\Xi: \Xi \to \Lambda$ is a coarsely Lipschitz map, with respect to the canonical metrics on $\Xi$ and $\Lambda$.
\end{lemma}
\begin{proof}
Fix a left-invariant proper metric $d$ on $\Xi^\infty$ and $d'$ on $\Lambda^\infty$. According to Lemma \ref{characterization-CL}, it suffices to show that for any $t> 0$ there exists an $s> 0$ such that whenever $\xi, \eta \in \Xi$ satisfy $d(\xi, \eta)\leq t$, then $d'(f_1(\xi),f_1(\eta))\leq s$. Before showing this, let us first notice a couple of facts.

First, the defect set $D(f) = \{f(y)^{-1}f(x)^{-1}f(xy) \mid x, y \in \Xi^\infty\}$ is finite, so we can define 
$$C\coloneqq \max_{z\in D(f)} d'(z, e_\Lambda) = \max_{w\in D(f)^{-1}} d'(w, e_\Lambda),$$
where the last equality holds due to the left-invariance of the metric $d'$.

Also, for any $x,y\in\Xi$, from $f_1(y)^{-1}f_1(x)^{-1}f(xy)=f(y)^{-1}f(x)^{-1}f(xy) \in D(f)$ it follows that 
\begin{equation}\label{DD}
f(xy)^{-1}f_1(x)f_1(y)\in D(f)^{-1}.
\end{equation} 
If we take $x=y=e_\Xi$, formula \eqref{DD} gives us $f_1(e_\Xi)\in D(f)^{-1}$, so
\begin{equation}\label{EE}
d'(f_1(e_\Xi),e_\Lambda)\leq C.
\end{equation}
If we take any $\xi, \eta \in \Xi$ and put  $x=\eta^{-1}$, $y=\xi$ into formula \eqref{DD}, we get
$f(\eta^{-1}\xi)^{-1}f_1(\eta^{-1})f_1(\xi)\in D(f)^{-1}$, which gives us
\begin{equation}\label{CC}
d'(f_1(\eta^{-1})f_1(\xi), f(\eta^{-1}\xi))=d'(f(\eta^{-1}\xi)^{-1}f_1(\eta^{-1})f_1(\xi),e_\Lambda)\leq C.
\end{equation}

Therefore, taking $\xi=\eta\in \Xi$ and using $f(\eta^{-1}\eta)=f_1(e_\Xi)$ and the inequalities \eqref{CC} and \eqref{EE}, we get
\begin{eqnarray}\label{EtaEta}
d'(f_1(\eta), f_1(\eta^{-1})^{-1}) &=& d'(f_1(\eta^{-1})f_1(\eta), e_\Lambda)\nonumber \\
   &\leq& d'(f_1(\eta^{-1})f_1(\eta), f(\eta^{-1}\eta)) + d'(f_1(e_\Xi), e_\Lambda)\leq 2C.
\end{eqnarray}

\medskip

Finally, let us take a random $t>0$ and fix it. 

Since the metric $d$ on $\Xi^\infty$ is proper, the closed ball $\overline{B}(e_\Xi, t)\coloneqq \{\xi\in\Xi^\infty\ | \ d(e_\Xi, \xi)\leq t\}$ is compact, so it is finite, and therefore $f(\overline{B}(e_\Xi, t))$ is finite in $\Lambda^\infty$, which means that there exists an $S>0$ such that $f(\overline{B}(e_\Xi, t))\subseteq \overline {B'}(e_\Lambda, S)\coloneqq \{\lambda \in \Lambda^\infty \ | \ d'(e_\Lambda, \lambda)\leq S\}$.

Therefore, whenever we take any $\xi,\eta\in \Xi$ satisfying $d(\xi,\eta)=d(\eta^{-1}\xi, e_\Xi)\leq t$, we have 
\begin{equation}\label{oo}
d'(f(\eta^{-1}\xi), e_\Lambda)\leq S.
\end{equation}

Thus, for any  $\xi,\eta\in \Xi$ satisfying $d(\xi,\eta)\leq t$, using  inequalities \eqref{CC}, \eqref{oo} and \eqref{EtaEta},   we have
\begin{eqnarray*}
d'(f_1(\xi),f_1(\eta)) &\leq & d'(f_1(\xi),f_1(\eta^{-1})^{-1}) + d'(f_1(\eta^{-1})^{-1},f_1(\eta)) \nonumber \\
 & = & d'(f_1(\eta^{-1})f_1(\xi),e_\Lambda) +  d'(f_1(\eta^{-1})^{-1},f_1(\eta)) \nonumber \\
&\leq& d'(f_1(\eta^{-1})f_1(\xi),f(\eta^{-1}\xi)) +d'(f(\eta^{-1}\xi),e_\Lambda) + d'(f_1(\eta^{-1})^{-1},f_1(\eta)) \nonumber \\
&\leq & C+S+2C=3C+S,
\end{eqnarray*}
which finishes the proof.
\end{proof}
\begin{corollary}\label{cor: cor-Lip}
If $G$ and $H$ are countable groups and $f:G\to H$ is a quasimorphism, then $f: G \to H$ is a coarsely Lipschitz map, with respect to the canonical metrics on $G$ and $H$.
\end{corollary}
\begin{proof}
Note that $(G,G)$ and $(H,H)$ are countable approximate groups and use Lemma \ref{coarsely-Lipschitz} on the quasimorphism $f:(G,G)\to (H,H)$.
\end{proof}

%%%%%%%%%%
\section{The main theorem}
Let us restate the main theorem, for convenience.
\begin{usethmcounterof}{main}
Let $(\Xi, \Xi^\infty)$ and $(\Lambda,\Lambda^\infty)$ be countable approximate groups and let $f:(\Xi, \Xi^\infty) \to (\Lambda,\Lambda^\infty)$ be a quasimorphism. Then 
$$\asdim \Xi \leq \asdim \Lambda + \asdim \left( f^{-1}\left( f(e_\Xi)\cdot D(f)^{-1}\cdot D(f)\right)\right), $$
where $D(f)$ denotes the defect set of $f$.
\end{usethmcounterof}

\begin{remark}
Since $f:(\Xi, \Xi^\infty) \to (\Lambda,\Lambda^\infty)$ need not be symmetric nor unital, $f_1(\Xi)=f(\Xi)$ need not be an approximate subgroup of $\Lambda^\infty$. Therefore the function $f_1=f|_\Xi :\Xi \to \Lambda$ need not be surjective, that is, $f_1(\Xi)$ need not be equal to $\Lambda$.
But since $\Lambda$ is an approximate subgroup of $\Lambda^\infty$ and so it is symmetric, we have that, for each $\xi\in \Xi$, from $f_1(\xi)\in f_1(\Xi)\subseteq \Lambda$ it follows that the inverse of $f_1(\xi)$, i.e., $f_1(\xi)^{-1}=f(\xi)^{-1}$ is contained in $\Lambda$. Also, both $f_1(e_\Xi)=f(e_\Xi)$ and $e_\Lambda$ are contained in $\Lambda$.
\end{remark}

\begin{proof}[Proof of Theorem \ref{main}]
Fix a left-invariant proper metric $d$ on $\Xi^\infty$ and $d'$ on $\Lambda^\infty$. 

We will use notation $B(\xi, r)$ and $B'(\lambda, r)$ for open balls with radius $r>0$ in $(\Xi^\infty,d)$ and $(\Lambda^\infty, d')$, respectively, centered at $\xi \in\Xi^\infty$, $\lambda\in \Lambda^\infty$. Note that by left-invariance of $d'$ we have $B'(\lambda,r)=\lambda B'(e_\Lambda,r)$, for all $\lambda \in \Lambda^\infty$.

By Lemma \ref{coarsely-Lipschitz}, the restriction $f_1=f|_\Xi:\Xi\to \Lambda$ is coarsely Lipschitz. Therefore,
by Corollary \ref{cor: AbstractHurewicz} it suffices to show that, for every $r>0$, the collection 
$$\mathbb{X}_r\coloneqq \{f_1^{-1}(B'(\lambda,r))\ | \ \lambda\in f_1(\Xi)\}=
\{f_1^{-1}(B'(\lambda,r)\cap f_1(\Xi))\ | \ \lambda\in f_1(\Xi)\}$$
satisfies the condition 
\begin{equation}\label{star}
\asdim \mathbb{X}_r  \overset{u}{\leq} \asdim f^{-1}\left( f(e_\Xi)\cdot D(f)^{-1}\cdot D(f)\right).
\end{equation}

First, fix a random $r>0$. Then, for every $\lambda\in f_1(\Xi)=f(\Xi)$, pick an element $\xi_\lambda\in f_1^{-1}(\lambda)\subseteq \Xi$.
Note that for any $\widetilde\lambda\in \Lambda$ we have  
\[(\widetilde\lambda \cdot B'(e_\Lambda,r))\cap\Lambda \subseteq \widetilde\lambda \cdot\left(B'(e_\Lambda,r)\cap\Lambda^2\right),\] since if $z=\widetilde\lambda b=\lambda'$, for some $b\in B'(e_\Lambda,r)$, $\lambda' \in\Lambda$, then $b=\widetilde\lambda^{-1}\lambda'\in \Lambda^2$.

Now, for any $\lambda\in f_1(\Xi)=f(\Xi)$, we have
\begin{eqnarray} \label{first long}
f_1^{-1}(B'(\lambda, r) \cap f_1(\Xi)) &=& f_1^{-1}(B'(\lambda, r) \cap \Lambda) =f_1^{-1}((\lambda\cdot B'(e_\Lambda, r)) \cap \Lambda) \nonumber \\
 &\subseteq& f_1^{-1}(\lambda\cdot (B'(e_\Lambda, r)\cap \Lambda^2)) \nonumber \\
&=& f_1^{-1}\left( f_1\left( \xi_\lambda\right)\cdot (B'(e_\Lambda, r)\cap \Lambda^2)\right) \nonumber \\
&\subseteq&  f^{-1}\left(f(\xi_\lambda)\cdot\left(B'(e_\Lambda, r) \cap \Lambda^2\right)\right) \nonumber
\\
&=& \{z \in \Xi^\infty \mid f(z) \in  f(\xi_\lambda)\cdot(B'(e_\Lambda, r) \cap \Lambda^2)\} \nonumber \\
&=& \{z \in \Xi^\infty \mid f(\xi_\lambda)^{-1}f(z) \in B'(e_\Lambda, r) \cap \Lambda^2\} \nonumber\\
&=& \{z \in \Xi^\infty \mid f(\xi_\lambda)^{-1}  f(\xi_\lambda^{-1})^{-1}f(\xi_\lambda^{-1}) f(z) \in B'(e_\Lambda, r) \cap \Lambda^2\}.
\end{eqnarray}

Note that since $f$ need not be symmetric,  $f(\xi_\lambda)^{-1}$ need not be equal to $f(\xi_\lambda^{-1})$, which is why we needed to introduce the last expression in the previous sequence of equations.

By \eqref{Drho-1} and \eqref{Drho1} we have 
$f(\xi_\lambda)^{-1}  f(\xi_\lambda^{-1})^{-1}\in D(f)f(e_\Xi)^{-1}$ and $f(\xi_\lambda^{-1}) f(z) \in f(\xi_\lambda^{-1}z)D(f)^{-1}$. Thus, for the last row in \eqref{first long} we can write, for some $d,d'\in D(f)$:
$$
f(\xi_\lambda)^{-1}  f(\xi_\lambda^{-1})^{-1}f(\xi_\lambda^{-1}) f(z)=d\cdot f(e_\Xi)^{-1} \cdot f(\xi_\lambda^{-1}z)\cdot d'^{-1}\in B'(e_\Lambda, r) \cap \Lambda^2,
$$
which means that $f(\xi_\lambda^{-1}z) \in f(e_\Xi)\cdot D(f)^{-1}\cdot (B'(e_\Lambda, r)\cap \Lambda^2)\cdot D(f)$.
Therefore $$\xi_\lambda^{-1} z \in f^{-1}\left(f(e_\Xi)\cdot D(f)^{-1}\cdot (B'(e_\Lambda, r)\cap \Lambda^2)\cdot D(f)\right), \text{so}$$ 
$$z\in \xi_\lambda\cdot f^{-1}\left(f(e_\Xi)\cdot D(f)^{-1}\cdot (B'(e_\Lambda, r)\cap \Lambda^2)\cdot D(f)\right).$$
Thus, by \eqref{first long} we have 
\begin{eqnarray}\label{almost}
f_1^{-1}(B'(\lambda, r) \cap f_1(\Xi)) &\subseteq& \{ z \in \Xi^\infty \mid f(\xi_\lambda)^{-1}f(z) \in B'(e_\Lambda, r) \cap \Lambda^2\} 
\nonumber \\
&\subseteq& \xi_\lambda\cdot f^{-1}\left(f(e_\Xi)\cdot D(f)^{-1}\cdot (B'(e_\Lambda, r)\cap \Lambda^2)\cdot D(f)\right).
\end{eqnarray}
Since left-multiplication by $\xi_\lambda$ yields an isometry of $\Xi^\infty$,  using properties of $\asdim$ in Lemma \ref{lem: asdim properties} we have just reduced proving \eqref{star} to proving that, for each $r>0$,
\begin{equation}\label{eq: main}
\asdim f^{-1}\left(f(e_\Xi)\cdot D(f)^{-1}\cdot (B'(e_\Lambda, r)\cap \Lambda^2)\cdot D(f)\right) \leq
\asdim  f^{-1}\left(f(e_\Xi)\cdot D(f)^{-1}\cdot D(f)\right).
\end{equation}
To see this, it is enough to show that for any $r>0$ there is an $R>0$ so that 
\begin{equation}\label{R nbhd}
f^{-1}\left(f(e_\Xi)\cdot D(f)^{-1}\cdot (B'(e_\Lambda, r)\cap \Lambda^2)\cdot D(f)\right) \subseteq N_R\left(f^{-1}\left(f(e_\Xi)\cdot D(f)^{-1}\cdot D(f)\right)\right),
\end{equation}
since, by Corollary \ref{cor: asdim prop},  $N_R\left(f^{-1}\left(f(e_\Xi)\cdot D(f)^{-1}\cdot D(f)\right)\right)$ and $ f^{-1}\left(f(e_\Xi)\cdot D(f)^{-1}\cdot D(f)\right)$ have the same asymptotic dimension.

First let us see, for any $r>0$, what the elements of $f(e_\Xi)\cdot D(f)^{-1}\cdot (B'(e_\Lambda, r)\cap \Lambda^2)\cdot D(f)$ are.
By properness of metric $d'$ on $\Lambda^\infty$, the ball $B'(e_\Lambda,r)$ is finite, so both $B'(e_\Lambda, r)\cap\Lambda$ and $B'(e_\Lambda, r)\cap\Lambda^2$ are finite and we may write, using $e_\Lambda\in \Lambda\subset \Lambda^2$,
\[
B'(e_\Lambda, r)\cap\Lambda=\{\lambda_1=e_\Lambda, \lambda_2, \ldots , \lambda_N\}\subset\{\lambda_1, \ldots, \lambda_N, \lambda_{N+1}, \lambda_{N+2}, \ldots , \lambda_{N+k}\}= B'(e_\Lambda, r)\cap\Lambda^2.
\]
We also know that $D(f)$ is finite, so we can write $D(f)=\{d_1, d_2, \ldots , d_m\}$. Now we can write 
\begin{eqnarray}\label{elements}
& & f(e_\Xi)\cdot D(f)^{-1}\cdot (B'(e_\Lambda, r)\cap \Lambda^2)\cdot D(f)= \nonumber \\
&=&
\{f(e_\Xi)d_i^{-1}\lambda_\ell d_j\mid d_i, d_j \in D(f),\lambda_\ell \in B'(e_\Lambda , r)\cap \Lambda^2\}\nonumber\\
&=&\{f(e_\Xi)\}\sqcup \{f(e_\Xi)d_i^{-1} d_j \mid d_i,d_j\in D(f), i\neq j\} \sqcup \nonumber \\
&\sqcup& \{f(e_\Xi) d_i^{-1} \lambda_\ell d_j \mid d_i,d_j\in D(f),  \lambda_\ell \in B'(e_\Lambda,r)\cap(\Lambda^2\setminus\{e_\Lambda\})\}.
\end{eqnarray}

Taking the preimage of this, we get
\begin{eqnarray}\label{elements preimage}
& & f^{-1}\left(f(e_\Xi)\cdot D(f)^{-1}\cdot (B'(e_\Lambda, r)\cap \Lambda^2)\cdot D(f)\right)= \nonumber \\
&=&
 f^{-1}(f(e_\Xi))\sqcup \left(\bigsqcup_{d_i,d_j\in D(f), i\neq j}f^{-1}(f(e_\Xi)d_i^{-1}d_j)\right) \sqcup \nonumber \\
&\sqcup& \left( \bigsqcup_{\substack{d_i,d_j\in D(f) \\ \lambda_\ell \in B'(e_\Lambda,r)\cap(\Lambda^2\setminus \{e_\Lambda\} )}}
f^{-1}(f(e_\Xi)d_i^{-1}\lambda_\ell d_j)\right)
\end{eqnarray}

Note that the first two parts of the union in \eqref{elements preimage} are already contained in $f^{-1}(f(e_\Xi) \cdot D(f)^{-1}\cdot D(f))$, so it remains to investigate $f^{-1}(f(e_\Xi)d_i^{-1}\lambda_\ell d_j)$, for $d_i,d_j\in D(f)$, $\lambda_\ell \in B'(e_\Lambda,r)\cap(\Lambda^2\setminus \{e_\Lambda\})$. Some of these preimages may be empty, but for the sake of the argument let us assume that none of them are empty, and let us choose a representative from each one, i.e., let
$$\xi_{ij\ell}\in f^{-1}(f(e_\Xi)d_i^{-1}\lambda_\ell d_j)\subseteq \Xi^\infty,  \text{ for } i,j\in\{1,\ldots , m\}, \ \ell\in\{2, \ldots , N+k\}.$$
Define $$R\coloneqq \max_{i,j,\ell} d(e_\Xi, \xi_{ij\ell}).$$
Note that if, for some $i,j,\ell$,  we have that $f^{-1}(f(e_\Xi)d_i^{-1}\lambda_\ell d_j)=\{\xi_{ij\ell} \}$, then
$$d(\xi_{ij\ell}, f^{-1}\left(f(e_\Xi)\cdot D(f)^{-1}\cdot D(f)\right))\leq d(\xi_{ij\ell}, e_\Xi)\leq R.$$
If $f^{-1}(f(e_\Xi)d_i^{-1}\lambda_\ell d_j)$ has more than one element, and if we take any $\eta_{ij\ell}\in f^{-1}(f(e_\Xi)d_i^{-1}\lambda_\ell d_j)\setminus\{\xi_{ij\ell}\}$, then by \eqref{Drho1} we have
$$f(\eta_{ij\ell}\xi_{ij\ell}^{-1})\in f(\eta_{ij\ell})f(\xi_{ij\ell}^{-1})D(f)=f(\xi_{ij\ell})f(\xi_{ij\ell}^{-1})D(f),$$
and applying \eqref{Drho1} again we get $f(\xi_{ij\ell})f(\xi_{ij\ell}^{-1})\in f(\xi_{ij\ell}\xi_{ij\ell}^{-1})D(f)^{-1}=f(e_\Xi)D(f)^{-1}$, so
$$f(\eta_{ij\ell}\xi_{ij\ell}^{-1})\in f(e_\Xi)\cdot D(f)^{-1}\cdot D(f).$$
Therefore 
$$d(\eta_{ij\ell}, f^{-1}\left(f(e_\Xi)\cdot D(f)^{-1}\cdot D(f)\right))\leq d(\eta_{ij\ell}, \eta_{ij\ell}\xi_{ij\ell}^{-1})=d(e_\Xi, \xi_{ij\ell}^{-1})
=d (\xi_{ij\ell},e_\Xi)\leq R.$$

Thusly each $f^{-1}(f(e_\Xi)d_i^{-1}\lambda_\ell d_j)$ from \eqref{elements preimage} is contained in $N_R\left(f^{-1}\left(f(e_\Xi)\cdot D(f)^{-1}\cdot D(f)\right)\right)$, so condition \eqref{R nbhd} is satisfied, which finishes the proof.
\end{proof}

As a special case of Theorem \ref{main}, we get:
\begin{corollary}\label{cor main}
If $f:G\to H$ is a quasimorphism of countable groups, then $$\asdim G\leq \asdim H +  \asdim\ \! \left(f^{-1}\left( f(e_\Xi)D(f)^{-1}D(f)\right)\right).$$
\end{corollary}

\begin{remark}
If $f:(\Xi,\Xi^\infty)\to(\Lambda,\Lambda^\infty)$ is a symmetric quasimorphism, i.e., if $f(x^{-1})=f(x)^{-1}$ for all $x\in \Xi^\infty$, then in the last line of expression \eqref{first long} in the proof of Theorem~\ref{main} we get $f(\xi_\lambda)^{-1}f(z)=f(\xi_\lambda^{-1})f(z)$, so the calculation that follows in the proof simplifies to $z\in \xi_\lambda\cdot f^{-1}\left((B'(e_\Lambda,r)\cap \Lambda^2)\cdot D(f)\right)$. This was used in the proof of Theorem \ref{Thm-symm-quasimorphism}, i.e.,  in \cite[Theorem 4.6]{Tonic} to show that for a symmetric and unital quasimorphism between countable approximate groups, one can put $\asdim f^{-1}(D(f))$ instead of 
$\asdim f^{-1}\left( f(e_\Xi)D(f)^{-1}D(f)\right)$ into the stated inequality of Theorem \ref{main}. Similarly, for a symmetric and unital quasimorphism of countable groups, one can make the same replacement in the formula of Corollary \ref{cor main}.
\end{remark}

\begin{remark}
Using definitions \ref{def: canonical coarse class} and \ref{def: asdim - group}, the expression $f^{-1}\left( f(e_\Xi)D(f)^{-1}D(f)\right)$ in Theorem \ref{main} and Corollary \ref{cor main} can be replaced by its canonical coarse class $\left[ f^{-1}\left( f(e_\Xi)D(f)^{-1}D(f)\right)\right]_c$.
\end{remark}

\medskip 

{\small Acknowledgment: This work has been supported by European Union – NextGenerationEU - \emph{Posebni problemi u algebarskim strukturama u teoriji brojeva i topologiji}.}

\end{document}